\documentclass[reqno,12pt]{amsart}
\usepackage{amsmath,amsthm,amssymb}

\date{}

\newtheorem{theorem}{Theorem}[section]
\newtheorem{lemma}[theorem]{Lemma}
\newtheorem{corollary}[theorem]{Corollary}
\newtheorem{proposition}[theorem]{Proposition}

\newtheorem{remark}[theorem]{Remark}

\numberwithin{equation}{section}

 \title{Moment measures and stability for Gaussian inequalities}

\author{Alexander~V.~Kolesnikov}
\address{Higher School of Economics, Moscow,  Russia}
\email{Sascha77@mail.ru}
\thanks{This research has been supported by
the Russian Science Foundation Grant N 17-11-01058 (at Moscow Lomonosov State University)
}

\author{Egor~D.~Kosov}
\address{Department of Mechanics and Mathematics, Moscow State University, 119991 Moscow, Russia;
Higher School of Economics, Moscow,  Russia}
\email{ked\_2006@mail.ru}
\thanks{The second author is a Young
Russian Mathematics award winner and would like to thank its sponsors and jury.}

\begin{document}

\maketitle

  \begin{abstract}
  Let $\gamma$ be the standard Gaussian measure on $\mathbb{R}^n$
  and let $\mathcal{P}_{\gamma}$  be the space of probability measures
  that are absolutely continuous   with respect to $\gamma$.
  We study lower bounds for the functional
  $\mathcal{F}_{\gamma}(\mu) = {\rm  Ent}(\mu) - \frac{1}{2} W^2_2(\mu, \nu)$,
 where $\mu \in \mathcal{P}_{\gamma}, \nu  \in \mathcal{P}_{\gamma}$, ${\rm  Ent}(\mu) = \int \log\bigl( \frac{\mu}{\gamma}\bigr) d \mu$ is the relative  Gaussian entropy, and $W_2$
 is the quadratic  Kantorovich distance.
 The minimizers of $\mathcal{F}_{\gamma}$ are solutions to a dimension-free Gaussian analog
 of the (real)  K{\"a}hler--Einstein equation.
 We show that  $\mathcal{F}_{\gamma}(\mu) $ is bounded from below under the assumption
 that the Gaussian Fisher information of $\nu$  is finite and
 prove a~priori estimates for the minimizers.
 Our approach relies on certain stability estimates
 for  the Gaussian log-Sobolev and Talagrand transportation inequalities.
  \end{abstract}


\section{Introduction}
Given a  probability measure $\nu = \varrho dx$ one can try to find
a log-concave measure $\mu = e^{-\Phi} dx$ (i.e., $\Phi$ is a convex function)
satisfying the following remarkable property:
$\nu$~is the image of $\mu$ under the mapping $T$
generated by the logarithmic gradient of~$\mu$:
$$
T(x) = \nabla \Phi(x),
\quad
\nu = \mu \circ T^{-1}.
$$
Following the terminology from \cite{CK},
 we say that  $\nu$ is a moment measure if such a function $\Phi$ exists.

There are many motivations to study moment measures. The associated equation on $\Phi$
$$
e^{-\Phi} = \varrho(\nabla \Phi) \det D^2 \Phi
$$
is a non-linear elliptic PDE of the Monge--Amp\`ere type.
After a suitable complexification it turns out to be
a particular case of the complex Monge--Amp\`ere equation.
The case where $\nu$ is Lebesgue measure
on a polytope with rational coordinates is of special interest
in differential and algebraic geometry because of its relation to the theory of toric varietes.
First results on the well-posedness of this equation have been established
in a series of geometric papers (see \cite{WangZhu}, \cite{BB},  \cite{CK},
 and the references therein).
The most general result on existence of the moment measure has been obtained in \cite{CK}
under fairly general assumptions. It   is known that
$\Phi$ is a maximum point of the following functional:
\begin{equation}
\label{prekopa}
J(f) = \log \int e^{-f^*} dx  - \int f d \nu,
\end{equation}
where $f^*$ is the Legendre transform of $f$.
This functional  has  deep relations to the classical Brunn--Minkowski theory.
In particular, $J$ is concave under the usual addition and this fact
is a particular form of the famous Pr\'ekopa--Leindler inequality.
The measure $\mu$ is unique up to translations. To determine it uniquely
we always assume that the barycenter (mean) of $\mu$ equals zero: $\int x d \mu=0$.

An alternative  viewpoint was suggested in \cite{S}, where another natural
functional was proposed.
It was shown in \cite{S} that $\rho = e^{-\Phi} $ gives a minimum to the functional
\begin{equation}
\label{KEfunc}
\mathcal{F}(\rho) = -\frac{1}{2} W^2_2(\nu, \rho dx)
+ \frac{1}{2} \int x^2 \rho \ dx + \int \rho \log \rho dx.
 \end{equation}
 Here $W_2$ is the Kantorovich distance for the cost function $c(x,y) = |x-y|^2$.
Unlike the approach of \cite{CK}, the moment measure problem is viewed
here as a problem on the space of probability measures
equipped with the quadratic Kantorovich distance. We emphasize
that the mass transportation problem is very relevant here.
Indeed, the mapping $x \to \nabla \Phi(x)$ is the optimal
transportation taking $e^{-\Phi} dx$ to $\nu$.
However, since $\mu$ depends on $\Phi$ explicitly,
there is no simple way to find $\Phi$ as a solution to a Monge--Kantorovich problem.

Following the idea from \cite{S}, we are looking for
the minima of the Gaussian analog of (\ref{KEfunc})
$$
\mathcal{F}_{\gamma}(\rho) = -\frac{1}{2} W^2_2(g \cdot \gamma, \rho \cdot \gamma)  + {\rm Ent} \rho,
$$
where $
\gamma= \frac{1}{(\sqrt{2 \pi})^n} e^{-\frac{|x|^2}{2}} \ dx,
 $
$$
{\rm Ent} \rho  =\int \rho \log \rho d \gamma
 $$
 is the Gaussian entropy of $g$.

This question is motivated by the following infinite-dimensional analog of the
  moment measure problem.
Let $\gamma $ be the standard  Gaussian product measure on $\mathbb{R}^{\infty}$
and let $\nu = g \cdot \gamma$  be a probability measure
such that
$$\int x_i g d \gamma=0 \quad \hbox{for every $i \in \mathbb{N}$.}
$$
The problem is to find a    log-concave measure  $\mu = e^{-\varphi} \cdot \gamma$
such that $\nu$ is the image of $\mu$ under the mapping
$$
T(x) = x + \nabla \varphi,
$$
where  $\nabla \varphi$ is the Cameron--Martin gradient.

There exists a rich  theory of optimal transportation on the
Wiener space with a number of interesting results
(see \cite{BoKo2005}, \cite{BoKo2011}, \cite{BoKo2012},
\cite{Cav},  \cite{FN}, \cite{FU}, and \cite{Kol04}).
So the well-posedness of the moment measure problem on the Wiener space
is a natural and interesting question. We emphasize that
the finite-dimensional estimates obtained in this paper are the first crucial
step towards infinite-dimensional spaces. However,
the infinite-dimensional moment measure problem
seems to be delicate and requires hard technical work. This will be done  in a forthcoming
paper of the authors.

The following theorem is our main result (see Theorem \ref{main}).

  \begin{theorem}
    Assume that $g$ is a probability density satisfying
  $
  I(g) < \infty,
  $
  where
   $$
   I(g)= \int \frac{|\nabla g|^2}{g} d \gamma
   $$
   is the Gaussian Fisher information of $g$.
   Then there exists a constant $C>0$ depending only on $I(g)$    such that
   $$
   \mathcal{F}_{\gamma} \ge -C
   $$
and
   $$
   W^2_2(g \cdot \gamma, \rho \cdot \gamma) \le C,
   $$
   where $\rho \cdot \gamma$ is the minimum point
   of $\mathcal{F}_{\gamma}$ satisfying the condition $\int x  \rho d \gamma=0$.
      \end{theorem}

Our approach is based on certain stability results for the log-Sobolev
 and the Talagrand transportation inequalities:
 $$
 \frac{1}{2} I(g) - {\rm Ent} g \ge \delta_1(g),$$
 $$  {\rm Ent} g  - \frac{1}{2} W^2_2(\gamma, g \cdot \gamma)
 \ge \delta_2(g),
 $$
 where $\delta_1, \delta_2$ are
 some non-negative functionals defined on probability densities.

 The stability of the Euclidean isoperimetric
 inequality (see the survey paper  \cite{Fig})
 and the Gaussian inequalities  (see \cite{IM},
 \cite{BGRS}, \cite{FIL}, \cite{CFP}, and \cite{LGP})
 has been recently studied by many researchers.
 In this paper we establish several new results in this direction and  give new simple proofs
 of some previously known inequalities.

Finally, we obtain  a priori estimates for the  (centered) minimum
point $\rho \cdot \gamma = e^{-\varphi} \cdot \gamma$ of $\mathcal{F}_{\gamma}$.
In particular, applying the approach developed in \cite{BoKo} for
the standard Monge--Kantorovich problem, we establish
new bounds for the entropy- and information-type  functionals
$$
\int \rho |\log \rho|^p d \gamma, \ \int \rho \Bigl| \frac{\nabla \rho}{\rho}\Bigr|^p d \gamma, \ p \ge 1,
$$
and certain exponential moments.

\section{Stability results}

\subsection{Notation}

We shall use some standard results and terminology from Gaussian analysis
(see \cite{BogGauss})
and optimal transportation theory (see \cite{BoKo2012}).

Let $\gamma$ be the standard Gaussian measure on $\mathbb{R}^n$:
$$
\gamma = \frac{1}{(\sqrt{2 \pi})^n} e^{-\frac{|x|^2}{2}} dx.
 $$
 We denote by $T$
 the optimal transportation taking  $g \cdot \gamma$ to $\gamma$.
 Recall that $T$  gives a minimum to the functional
 $$
F \to  \int | F(x) - x |^2  g d \gamma
 $$
 considered on the mappings taking $g \cdot \gamma$ to $\gamma$.

Moreover, $T$ is the gradient of a convex function.
It can be written in the form
$$
T(x) = x + \nabla \varphi(x),
$$
where   the potential $\varphi$ satisfies the estimate
$$
D^2 \varphi \ge - \rm{Id}.
$$
The corresponding Kantorovich distance $W_2(\gamma, g \cdot \gamma)$ for the cost function
$c(x,y) = |x-y|^2$ can be computed as follows:
$$
W^2_2(\gamma, g \cdot \gamma) = \int |\nabla \varphi|^2 g d \gamma.
$$
The  notation
$$\| A\| = \sqrt{ {\rm Tr} (AA^T)}$$
will be used for the Hilbert--Schmidt norm of the matrix $A$
and $\| A\| _{op}$ will denote  the operator norm.
We also use the standard notation for the (Gaussian) entropy
$$
{\rm Ent} g = \int g \log g d \gamma
$$
and  information
$$
I(g) = \int \frac{|\nabla g|^2}{g} d \gamma.
$$

\subsection{Stability for  the logarithmic Sobolev inequality}

The celebrated Gaussian  logarithmic Sobolev inequality
\begin{equation}
\label{LSI}
\frac{1}{2} I(g) \ge {\rm Ent} g
\end{equation}
is one of the central results in  Gaussian analysis.
Here $g$ is a  sufficiently regular probability density.
For  the proofs and the history of (\ref{LSI}), see \cite{BogGauss}, \cite{Led}, and \cite{BGL}.

Throughout the paper we assume that $g$ has finite information.

{\bf Assumption I.}
$$
I(g) < \infty.
$$

It is known that (\ref{LSI}) is sharp and the corresponding minimizers have the form
$g = e^{l}$, where $l$ is an affine function. This has been proved by Carlen in \cite{Carlen}.
He has shown that the so-called log-Sobolev deficit
$$
\frac{1}{2} I(g) - {\rm Ent} g
$$
is bounded from below by a non-negative term,
which is a functional involving certain integral transform of $g$.

Yet another representation has been obtained in  \cite{Kol}:
\begin{theorem} (\cite{Kol})
Let $
T = x + \nabla \varphi
$
be the optimal transportation taking  $g \cdot \gamma$ nto $\gamma$, where
$g$ is a sufficiently regular probability density.
Then the following representation holds:
\begin{align}
\label{BochGauss}
I(g)
& = 2  {\rm Ent} g
+ 2 \int \bigl( \Delta \varphi - \log \det(I +  D^2 \varphi) \bigr) g d \gamma
 + \int \| D^2 \varphi\|^2 g d \gamma \nonumber \\& + \int \sum_{i=1}^n {\rm Tr} \Bigl[ ({\rm Id} + D^2 \varphi)^{-1} (D^2 \varphi_{x_i} ) \Bigr]^2 g d\gamma.
\end{align}
\end{theorem}

\begin{remark}
\rm
{\bf (Regularity of $\varphi$)}. The gradient of $\varphi$
is well-defined almost everywhere because $\frac{|x|^2}{2} + \varphi $ is a convex function.
Identity (\ref{BochGauss}) ensures that $\varphi$ belongs to an appropriate second-order  Sobolev space
(see \cite{BoKo2011} for details). The reader can always assume that $g$
is bounded away from zero and locally smooth; this implies the local smoothness of $\varphi$
(see \cite{Kol} and \cite{BoKo2011}). In almost all our statements the minimal
assumption about  $g$ is $I(g) < \infty$. This case follows easily  from the case of
a smooth potential by the standard approximation procedure.
\end{remark}

It is important to mention that all the terms in the right-hand side of (\ref{BochGauss})
are non-negative.
This result is closely related to the so-called Gaussian stability inequalities,
which have been recently investigated in a series of papers  \cite{IM}, \cite{BGRS}, \cite{FIL}, \cite{CFP}.
Technically speaking, these are estimates of the type
$$
\frac{1}{2} I(g) - {\rm Ent} g  \ge F(g \cdot \gamma),
$$
where $F$ is  a non-negative functional on probability densities (measures).
In our work we apply  other well-known results deeply connected with~(\ref{LSI}):
the Gaussian Talagrand  transportation inequality
\begin{equation}
\label{Tal}
{\rm Ent} g  \ge \frac{1}{2} W^2_2(\gamma, g \cdot \gamma)
\end{equation}
and the HWI inequality
\begin{equation}
\label{HWI}
\frac{1}{2} W^2_2(g \cdot \gamma, \gamma)
+ {\rm Ent} g  \le \sqrt{ I(g) } W_2(\gamma, g \cdot \gamma)
\end{equation}
(see \cite{Led} and \cite{BGL}).

The Talagrand inequality and the HWI  inequality follow from the following
identity which is widely used
in transportation inequalities (see \cite[Theorem 9.3.1]{BGL}).

\begin{theorem}
Let $T(x) = x + \nabla u(x)$ be the optimal transportation
taking $g \cdot \gamma$ to $f \cdot \gamma$. Then
\begin{equation}
\label{ent-iden}
{\rm Ent}f = {\rm Ent} g +  \frac{1}{2} W_2^2(g\cdot\gamma, f \cdot\gamma)
+ \int \langle \nabla u, \nabla g \rangle d \gamma
+
\int \bigl( \Delta u- \log\det(I+D^2 u) \bigr) \, g d\gamma.
\end{equation}
\end{theorem}
To prove (\ref{Tal}) we set $g=1$ and use that $\Delta u- \log\det(I+D^2 u)$ is non-negative.
For (\ref{HWI}) we set $f=1$ and apply the Cauchy inequality.

In the proof of our main result we apply the following theorem from \cite{FIL}.

\begin{theorem}
 \label{FILest}
(\cite{FIL},  Theorem 1). Assume that $\nu = g \cdot \gamma$ satisfies the Poincar{\`e} inequality
$$
\int \Bigl( f - \int f g d \gamma\Bigr)^2 g d \gamma \le C_P \int |\nabla f|^2 g d \gamma
$$
and $\int x g d \gamma=0$. Then the following inequality holds:
\begin{equation}
\label{stab-opt}
\frac{1}{2}I (g) -
{\rm Ent} g  \ge \frac{1}{2} \frac{C_P \log C_P - C_P +1}{(C_P-1)^2} I(g).
\end{equation}
\end{theorem}

\begin{remark}
\rm
Stability estimates of the same type,
 but with non-sharp constants  can be derived  from (\ref{BochGauss}) and
(\ref{HWI}).
Let $\nu$ satisfy the assumptions of Theorem \ref{FILest}. Then
$$
I(g)  \ge 2 {\rm Ent} g   + \frac{1}{C_P} W^2_2(\gamma, g \cdot \gamma).
$$
Indeed, the result follows immediately from (\ref{BochGauss})
and the following computations (we use the Poincar{\`e}
inequality and the change of variables formula)
$$
\int \|D^2 \varphi\|^2 g d \gamma = \sum_{i=1}^n \int |\nabla \varphi_{x_i} |^2 g d \gamma
\ge \frac{1}{C_P} \int \varphi^2_{x_i} g d \gamma  - \frac{1}{C_P}\Bigl(\int \varphi_{x_i} g d \gamma \Bigr)^2,
$$
$$
\int \varphi_{x_i} g d \gamma  = \int (T_i - x_i) g d \gamma  = \int x_i d \gamma -  \int x_i g d \gamma =0.
$$

Applying (\ref{HWI}) we get the following estimate for arbitrary $K \ge 1$:
$$
\tfrac{1}{2}W_2^2(g\cdot\gamma, \gamma) + {\rm Ent} g  \le \sqrt{I(g)}W_2(g\cdot\gamma, \gamma)
\le
\tfrac{1}{2} \Bigl(KW_2^2(g\cdot\gamma, \gamma) + \frac{1}{K} I(g) \Bigr).
$$
Hence
$$
{\rm Ent} g \le \frac{1}{2K} I(g) + \frac{K-1}{2} W_2^2(g\cdot\gamma, \gamma)
\le  \frac{1}{2K} I(g) + \frac{C_P(K-1)}{2} ( I(g) -  2{\rm Ent }).
$$
Equivalently,
$
{\rm Ent} g \le \frac{1}{2} I(g) \Bigl( \frac{\frac{1}{K} + (K-1)C_P}{1+ (K-1)C_P}\Bigr).
$
Choosing the optimal value of $K$, which is  $K = 1 + \frac{1}{\sqrt{C_P}}$, one gets
\begin{equation}
\label{MainStab-1}
I(g) - 2{\rm Ent} g \ge\frac{1}{(1 + \sqrt{C_P})^2} I(g).
\end{equation}

Note that this is a result of the same type as  in Theorem \ref{FILest},
but for large values of $C_P$ the constant in the right-hand
side of (\ref{stab-opt}) is of order
$\frac{\log C_P}{C_P}$, which is stronger than our result. We observe that the proof of
(\ref{MainStab-1}) modulo (\ref{BochGauss}) is easier than the proof of (\ref{stab-opt}),
but we  do not know how to deduce (\ref{stab-opt})  from (\ref{BochGauss}).
\end{remark}

We now prove another stability result  (\ref{igncp})
similar to (\ref{stab-opt}). In particularly,
both estimates are sharp: the equalities hold
for $g = \lambda^{\frac{n}{2}} e^{\frac{(1-\lambda)}{2} |x|^2}$.
Note, however, that under the assumption $I(g)< \infty$ the right-hand side of (\ref{stab-opt}) is always finite
and dimension-free, which is not the case for (\ref{igncp}).
This fact has rather unexpected interesting consequences (see Remark \ref{infdPoin}).

\begin{theorem}
Assume that $\nu = g \cdot\gamma$  satisfies the  Poincar{\'e} inequality
$$
\int f^2\, d\nu \le C_P \int|\nabla f|^2\, d\nu, \ \int f d \nu=0
$$
and $C_P \le  1$. Then
\begin{equation}
\label{igncp}
\frac{1}{2}I(g) - {\rm Ent} g \ge \frac{n(C_P \log C_P - C_P +1)}{2 C_P}.
\end{equation}
\end{theorem}

\begin{remark}\label{infdPoin}
\rm
Applying  (\ref{igncp}) to the infinite-dimensional ($n=\infty$)
 Gaussian measure $\gamma$, we obtain the following result:
if $g \cdot \gamma$ satisfies $I(g) < \infty$ and admits a finite Poincar{\'e} constant $C_P$
(for the Cameron--Martin norm), then $C_P \ge 1$. We believe that
the assumption $I(g) < \infty$ is unnecessary for this observation
and can be relaxed.

For the proof, we
write the right-hand side as  $\frac{n\Delta(C^{-1}_P- 1)}{2}$, where $\Delta(t)  = t - \log(1+t)$.
Let us apply (\ref{BochGauss}). Let  $\lambda_i$, $i\in\{1,\ldots, n\}$  be the eigenvalues of $D^2\varphi$.
Then
$$
\tfrac{1}{2}\|D^2\varphi\|_{HS}^2 + \Delta\varphi - \log\det(I+D^2\varphi)
=
\sum_{i=1}^n \tfrac{1}{2}\lambda_i^2 + \lambda_i - \log(1+\lambda_i)
=
\tfrac{1}{2}\sum_{i=1}^n \Delta((\lambda_i+1)^2 - 1).
$$
Note that
$$
\Delta^*(s):=\sup_{t\ge-1}\{st - \Delta(t)\} =
-s - \log(1-s), \quad s\le1.
$$
Thus,
\begin{multline*}
\tfrac{1}{2}\|D^2\varphi\|_{HS}^2 + \Delta\varphi - \log\det(I+D^2\varphi)
\ge
\tfrac{1}{2}\sum_{i=1}^n \bigl[s((\lambda_i+1)^2 - 1) - (-s - \log(1-s))\bigr]
\\
=
\tfrac{1}{2}s\sum_{i=1}^n (\lambda_i+1)^2 + \tfrac{1}{2}n\log(1-s)
=
\tfrac{1}{2}s\|I+D^2\varphi\|_{HS}^2 + \tfrac{1}{2}n\log(1-s).
\end{multline*}
Applying the relations
$$\|I+D^2\varphi\|_{HS}^2 = \sum_{i=1}^n|\nabla (x_i + \varphi_{x_i})|^2,
$$
$$
\int [x_i + \varphi_{x_i}]\, g d\gamma = \int y_i d\gamma = 0,
$$
we obtain for $ s\ge 0$
\begin{multline*}
\frac{1}{2}I(g) - {\rm Ent} g \ge
\frac{s}{2 C_P}\sum_{i=1}^n\int|x_i + \varphi_{x_i}|^2\, g d\gamma + \frac{1}{2}n\log(1-s)
\\=
\frac{s}{2 C_P}\sum_{i=1}^n\int|y_i|^2\, d\gamma + \tfrac{1}{2}n\log(1-s)
=
\frac{n}{2} \Bigl(\frac{s}{C_P}+ \log(1-s)\Bigr).
\end{multline*}
Taking $s = 1-C_P$ we obtain the desired result.
\end{remark}

Identity (\ref{BochGauss}) implies another estimate obtained earlier in
\cite[Theorem 1.1]{BGRS}.

\begin{corollary}
There holds the inequality
$$
I(g) - 2 {\rm Ent} g
\ge n \Delta \Bigl( \frac{1 }{n}\int \Bigl| \frac{\nabla g}{g} - x\Bigr|^2 g d \gamma  - 1\Bigr).
$$
\end{corollary}
\begin{proof}
Rewrite (\ref{BochGauss}) in the following way:
\begin{align*} I(g) &=   2 \mbox{\rm Ent}_{\gamma} g
+
 \int \Bigl( \|D^2 \Phi \|^2 - n - \log \det (D^2 \Phi)^2 \Bigr) \  g d \gamma
\\& +
\sum_{i=1}^{n} \bigl\| (D^2 \Phi)^{-\frac{1}{2}} D^2 \Phi_{x_i}  (D^2 \Phi)^{-\frac{1}{2}} \bigr\|^2 \ g d \gamma.
\end{align*}
By another result of \cite[Section 5]{Kol} we have
\begin{equation}
\label{0701}
\int \Bigl| \frac{\nabla g}{g} - x \Bigr|^2 g \ d \gamma = \int \| D^2 \Phi\|^2 g d  \gamma
+ \sum_{i}
\int \bigl\| (D^2 \Phi)^{-\frac{1}{2}} D^2 \Phi_{x_i}  (D^2 \Phi)^{-\frac{1}{2}} \bigr\|^2 \ g d\gamma.
\end{equation}
These two identities imply that
$$
\int \Bigl| \frac{\nabla g}{g} \Bigr|^2 g \ d \gamma=   2 \mbox{\rm Ent}_{\gamma} g
+
 \int \Bigl( \Bigl| \frac{\nabla g}{g} - x \Bigr|^2  - n - \log \det (D^2 \Phi)^2 \Bigr) \  g d \gamma.
$$
Using Jensen's inequality and
(\ref{0701}) we obtain
$$
- \int \log \det (D^2 \Phi)^2  \  g d \gamma
\ge -n \log \int  \frac{\|D^2 \Phi\|^2}{n}  \  g d \gamma
\ge - n \log \int  \frac{1}{n} \Bigl| \frac{\nabla g}{g}-x \Bigr|^2  \  g d \gamma.
$$
Hence
$$
\int \Bigl| \frac{\nabla g}{g} \Bigr|^2 g \ d \gamma \ge 2 \mbox{\rm Ent}_{\gamma} g
+
 \int \Bigl( \Bigl| \frac{\nabla g}{g} - x \Bigr|^2  - n \Bigr) \  g d \gamma
-n \log \int  \frac{1}{n} \Bigl| \frac{\nabla g}{g}-x \Bigr|^2  \  g d \gamma,
$$
which completes the proof.
\end{proof}

Certain stability estimates can be obtained
under (one-sided) uniform bounds
on the Hessian of the logarithmic potential
$-\log g$. The proof is based on the Caffarelli
contraction theorem  (see \cite{Kol-contr} and the references therein,
some new developments for higher order derivatives  can be found in \cite{KK}).

\begin{proposition}
Assume that
$$\int x g d \gamma =0$$
and
$${\rm Id} - D^2 \log g \ge \varepsilon \cdot{\rm Id}$$
for some constant $\varepsilon>0$. Then there exists a universal constant $c$  such that
$$
{\rm Ent} g \ge \Bigl(\frac{1}{2} + c \sqrt{\varepsilon} \Bigr) W^2_2(\gamma, g \cdot \gamma).
$$
\end{proposition}
\begin{proof}
Let $S(x) = x+\nabla\psi$ be the optimal transportation taking $\gamma$ to $g\cdot\gamma$.
According to the Caffarelli contraction theorem
$$
I + D^2 \psi \le \frac{1}{\sqrt{\varepsilon}}.
$$
Hence
$
\Delta \psi - \log \det D^2 (I + D^2 \psi) \ge c \sqrt{\varepsilon} \| D^2 \psi\|^2.
$
Then it follows from (\ref{ent-iden}) that
$$
{\rm Ent}g\ge \frac{1}{2} W^2_2(\gamma, g \cdot \gamma) + c \sqrt{\varepsilon}\int \| D^2 \psi\|^2 d \gamma.
$$
By the Gaussian Poincar{\'e} inequality
$$
\int \| D^2 \psi\|^2 d \gamma = \sum_{i=1}^n \int |\nabla \psi_{x_i}|^2  d \gamma
\ge  \sum_{i=1}^n \int  \psi^2_{x_i}  d \gamma =
 W^2(g\cdot \gamma, \gamma),
$$
which completes the proof.
\end{proof}

We end this subsection with an extension of Theorem \ref{FILest}
under the stronger assumption that $g \cdot \gamma$ satisfies
the log-Sobolev inequality.
Roughly speaking, the log-Sobolev deficit can be estimated from below
by
$$
\frac{1}{C_{LSI}} K_{a_\nu},
$$
where $C_{LSI}$ is the constant in the log-Sobolev inequality and $K_{a_\nu}$
is  the minimum of the Kantorovich functional
with a cost finction $c$ satisfying
$c(x) \sim x^2 \log x^2$ for large values of $x$ and $c(W_2(\nu,\gamma))=0$.

\begin{theorem}
Assume that  $\nu = g\cdot\gamma$ satisfies the logarithmic Sobolev inequality
$$
\int f^2\log f^2\, d\nu - \int f^2 d \nu \cdot \log \Bigl( \int f^2 d \nu \Bigr)  \le C_{LSI} \int|\nabla f|^2\, d\nu.
$$
Then
$$
\frac{1}{2}I(g) - {\rm Ent g}  \ge \frac{1}{ C_{LSI}} K_{a_\nu}(\nu, \gamma),
$$
where $a_\nu = W_2(\nu, \gamma)$ and $K_a(\nu, \gamma)$  is the minimum of the Kantorovich functional corresponding to the cost function
$$
c_a(x) = a^2\Bigl(1-\frac{|x|^2}{a^2}+\frac{|x|^2}{a^2}\log\frac{|x|^2}{a^2}\Bigr).
$$
\end{theorem}
\begin{proof} Let $T(x) = x + \nabla\varphi$ be
the optimal transportation taking $\nu$ to $\gamma$.
We apply formula (\ref{BochGauss})
and estimate the integral of $\|D^2\varphi\|^2_{\rm HS}$ from below:
\begin{multline*}
C_{LSI} \int \|D^2\varphi\|^2_{\rm HS}\, d\nu =
\sum_{j=1}^n C_{LSI}\int |\nabla\varphi_{x_j}|^2\, d\nu
\ge
\sum_{j=1}^n \int \varphi_{x_j}^2 \log \Bigl( \frac{\varphi_{x_j}^2}{\int \varphi^2_{x_j} d \nu} \Bigr)\, d\nu
\\
=
W_2^2(\nu, \gamma)
\int \sum_{j=1}^n \alpha_j
 \frac{\varphi_{x_j}^2}{\int \varphi^2_{x_j} d \nu} \log\Bigl( \frac{\varphi_{x_j}^2}{\int \varphi^2_{x_j} d \nu} \Bigr)\, d\nu,
\end{multline*}
where $\alpha_j = \frac{\int \varphi^2_{x_j} d \nu}{W_2^2(\nu, \gamma)}$.
The function  $t\mapsto t\log t$ is convex for $t>0$ and $\sum_{j=1} \alpha_j=1$.
Hence the above expression is not less than
\begin{multline*}
W_2^2(\nu, \gamma)\int \frac{\sum_{j=1}^n \varphi_{x_j}^2}{W_2^2(\nu, \gamma)}
\log \Bigl(\frac{\sum_{j=1}^n \varphi_{x_j}^2}{W_2^2(\nu, \gamma)}\Bigr)\, d\nu
\\
=
\int |\nabla\varphi|^2
\biggl[\log \Bigl(\frac{|\nabla\varphi|^2}{W_2^2(\nu, \gamma)}\Bigr)-1 + \frac{W_2^2(\nu, \gamma)}{|\nabla\varphi|^2}\biggr]\, d\nu
\ge
K_{a_\mu}(\nu, \gamma),
\end{multline*}
which completes the proof.
\end{proof}

\subsection{Stability for  the Talagrand transportation inequality}

The aim of the following proposition is to give a simplified proof
of another result from \cite[Theorem 5]{FIL} with a more precise constant.

\begin{lemma}
The function
$\Delta(t) = t - \log(1+t)$, $t >-1$ has the following properties:
\begin{enumerate}
\item $\Delta(t)$  is convex,
\item $\Delta(\sqrt{t})$ is concave on $[0,+\infty)$ and, in particular, subadditive,
\item $\Delta(t)\ge\Delta(|t|)$,
\item $\Delta(t)\ge(1-\log2)\min(t,t^2)$ on  $[0,+\infty)$.
\end{enumerate}
\end{lemma}

\begin{proposition}
Assume that $$\int x g d \gamma=0.$$
The deficit ${\rm Ent}g- \tfrac{1}{2} W_2^2(g\cdot\gamma, \gamma)$
of the Talagrand transportation inequality satisfies the following estimate:
\begin{multline*}
{\rm Ent}g- \tfrac{1}{2} W_2^2(g\cdot\gamma, \gamma)
\ge
\Delta\bigl(\tfrac{1}{2}n^{-1/2}W_{1,1}(g\cdot\gamma, \gamma)\bigr)
\\
\ge
(1-\log2)\min\bigl\{\tfrac{1}{2}n^{-1/2}W_{1,1}(g\cdot\gamma, \gamma), \tfrac{1}{4}n^{-1}W^2_{1,1}(g\cdot\gamma, \gamma)\bigr\}
\end{multline*}
where $W_{1,1}$  is the transportation cost corresponding to $c(x,y) = \sum_{i=1}^n |x_i-y_i|$.
\end{proposition}
\begin{proof}
Let $S(x) = x+\nabla\psi$ be the optimal transportation taking  $\gamma$ to $g\cdot\gamma$
and let $\lambda_i$ be all eigenvalues of $D^2 \varphi$.
Applying (\ref{ent-iden}) we obtain
\begin{multline*}
{\rm Ent}g- \tfrac{1}{2} W_2^2(g\cdot\gamma, \gamma)
=
\int \sum_{i=1}^n \lambda_i - \log(1+\lambda_i)\, d\gamma
\\
=
\int \sum_{i=1}^n \Delta(\lambda_i)\, d\gamma
\ge
\int \sum_{i=1}^n \Delta(|\lambda_i|)\, d\gamma
=
\int \sum_{i=1}^n \Delta\bigl(\sqrt{\lambda_i^2}\bigr)\, d\gamma
\ge
\int \Delta\Bigl(\bigl[\sum_{i=1}^n\lambda_i^2\bigr]^{1/2}\Bigr)\, d\gamma
\\
=
\int \Delta\bigl(\|D^2\psi\|_{\rm HS}\bigr)\, d\gamma
\ge
\Delta\Bigl(\int\|D^2\psi\|_{\rm HS}\, d\gamma\Bigr).
\end{multline*}
Now we note that
\begin{multline*}
\int\|D^2\psi\|_{\rm HS}\, d\gamma
=
\int\bigl(\sum_{i=1}^n |\nabla\psi_{x_i}|^2\bigr)^{1/2}\, d\gamma
\ge
n^{-1/2}\int\sum_{i=1}^n |\nabla\psi_{x_i}|\, d\gamma
\\
\ge
\tfrac{1}{2}n^{-1/2}\int\sum_{i=1}^n |\psi_{x_i}|\, d\gamma
\ge
\tfrac{1}{2}n^{-1/2}W_{1,1}(g\cdot\gamma, \gamma),
\end{multline*}
where we apply
the equality
$$
\int\psi_{x_i} \, d\gamma = \int (x_i+\psi_{x_i}) \, d\gamma = \int x_i g\, d\gamma = 0
$$
and the $L^1$-Poincar{\'e} (Cheeger) inequality for $\gamma$:
$$
\int \bigl|f - \int f d \gamma\bigr| d \gamma \le 2 \int |\nabla f| d \gamma,
$$
which completes the proof.
\end{proof}

\section{A priori estimates for the K{\"a}hler--Einstein equation}

A moment mesure on  $\mathbb{R}^n$ is a probability measure  $\nu$ on $\mathbb{R}^n$
that is the image of
another probability measure $\mu = e^{-\Phi} dx$ under the
mapping $x \to \nabla \Phi(x)$, where $\Phi$ is a convex function.
If  $\nu$  admits a smooth density $\varrho$, then $\Phi$
solves the K{\"a}hler--Einstein equation
$$
\varrho(\nabla \Phi) \det D^2 \Phi = e^{-\Phi}.
$$
It was shown in   \cite{CK} that every measure $\nu$ with zero mean satisfying
the condition $\nu(L)=0$ for any subspace $L$ of dimension
less than $n$ is a moment measure.
The function $\Phi$ is uniquely determined up to a translation.

We will be interested in the following Gaussian analog of the  K{\"a}hler--Einstein equation:
given a probability measure
$$
\varrho \ dx = g \cdot \gamma,
$$
find $\varphi$ such that $g \cdot \gamma$ is the image of  the log-concave probability measure
$$
\rho \cdot \gamma = e^{-\varphi} \cdot \gamma
$$
under the mapping
$$T(x) =  x + \nabla \varphi(x).
$$

Clearly, there is a simple connection
between this problem and the ''Euclidean'' moment measure problem.
Namely, $\Phi$ and $\varphi$ are related by the following formula:
$$
\Phi(x) = \frac{|x|^2}{2} + \varphi(x) + \frac{n}{2} \log 2 \pi.
$$
However, the Gaussian modification of the  moment measure problem
is meaningful in any infinite-dimensional space equipped with a Gaussian measure.

Since $\Phi$ is unique up to a translation, it will be natural
to impose the following requirement that determines $\varphi$ uniquely.

{\bf Assumption II.} The measure $ \rho \cdot \gamma = e^{-\varphi} \cdot \gamma$ satisfies
the condition
$$
\int x_i e^{-\varphi} d \gamma =0 \quad \forall\, i.
$$

The existence and uniqueness of $\varphi$ follows from the results of \cite{CK}.
It follows from the main result of \cite{S} that $e^{-\varphi}$
 gives a minimum to the following functional:
$$
\mathcal{F}_{\gamma}(\rho) = -\frac{1}{2} W^2_2(g \cdot \gamma, \rho \cdot \gamma)  + \int \rho \log \rho \ d \gamma.
$$
We wish to find a condition on $g$ which guarantees  that $F$ is bounded
from below by a dimension-free functional depending on $g$.

\subsection{Information controls $\mathcal{F}$}

Throughout this subsection
$\rho \cdot \gamma = e^{-\varphi} d \gamma$ is the (unique)
minimum point of $\mathcal{F}_{\gamma}$ with zero mean.

\begin{lemma}
\label{distcontrol}
Assume that the measure
$\rho \cdot \gamma = e^{-\varphi} d \gamma$ satisfies
the inequality
$$
{\rm Ent} \rho \le \frac{1-\delta}{2} I(\rho)
$$
for some $0 < \delta <1$.
Then
$$
W_2(\rho \cdot \gamma,  g \cdot \gamma)
\le
\frac{1+ \sqrt{1-\delta}}{\delta} W_2(g \cdot \gamma,  \gamma).
$$
\end{lemma}
\begin{proof}
We have
$$
 W^2_2(\rho \cdot \gamma,  \gamma) \le
 2 {\rm Ent}(\rho) \le (1-\delta) I(\rho)
 = (1-\delta) W^2_2(\rho \cdot \gamma,  g \cdot \gamma).
$$
Hence by the triangle inequality
$$
W_2(\rho \cdot \gamma, g \cdot \gamma)  \le W_2(\rho \cdot \gamma, \gamma) + W_2(g \cdot \gamma,  \gamma)
\le \sqrt{1-\delta} W_2(\rho \cdot \gamma, g \cdot \gamma) + W_2(g \cdot \gamma,  \gamma) ,
$$
which completes the proof.
\end{proof}

\begin{theorem}
\label{poin} There exists a pair of universal constants $C_1, C_2$ such that
$$
C_P\le \max\{C_1, \exp(C_2 I(g))\},
$$
where $C_P$ is the Poincar{\'e} constant of the measure $e^{-\varphi} \cdot \gamma$.
\end{theorem}
\begin{proof}
Let $x + \nabla \psi$ be the optimal transportation taking  $g \cdot \gamma$
to $e^{-\varphi} \cdot \gamma$. It is well-known that
$$
x + \nabla \psi = T^{-1} \quad \hbox{$g \cdot \gamma$-a.e.}
$$
 and
$$
W^2_2(\rho \cdot \gamma, g \cdot \gamma) =
\int |\nabla \varphi|^2 e^{-\varphi} d \gamma = \int |\nabla \psi|^2 g d \gamma.
$$
First we note that  $e^{-\varphi} \cdot \gamma$
is a log-concave measure, hence
it has finite moments of all orders and a finite Poincar{\'e}
constant $C_P < \infty$  (see \cite[Theorem 4.6.3]{BGL}).

 Note that
$$
I(\rho) = W_2(\rho \cdot \gamma,  g \cdot \gamma)
\le W_2(\rho \cdot \gamma, \gamma) + W_2(g \cdot \gamma, \gamma).
$$
The right-hand side of this inequality is finite, because
$g \cdot \gamma$ and $\rho \cdot \gamma$ have finite second moments.
Thus,  $I(\rho) < \infty$.
 Moreover, approximating $g$
by smooth densities with uniformly  bounded second derivatives of $\log g$
we can assume without loss of generality that
$\nabla \psi$ is globally Lipschitz (see Theorem \ref{sdb}).

It follows from the previous lemma and Theorem \ref{FILest}
that
$$
W_2(\rho \cdot \gamma,  g \cdot \gamma)
\le
\frac{1+ \sqrt{1-\delta}}{\delta}W_2(g \cdot \gamma,  \gamma),
$$
where
$
\delta =  \frac{C_P \log C_P - C_P +1}{(C_P-1)^2}.
$
Applying (\ref{ent-iden}) we obtain
\begin{align*}
\int \bigl( \Delta \psi -  \log \det  \bigl( {\rm Id} + D^2 \psi \bigr) \bigr) g d \gamma  & + \int g \log g d \gamma  +  \frac{1}{2}\int |\nabla \varphi|^2 e^{-\varphi} d \gamma
\\& =  - \int \varphi e^{-\varphi} d\gamma - \int \langle \nabla \psi, \nabla g \rangle d \gamma.
\end{align*}
By the log-Sobolev inequality
$$- \int \varphi e^{-\varphi} d\gamma
\le \frac{1}{2}\int |\nabla \varphi|^2 e^{-\varphi} d \gamma.
$$
Hence
\begin{align}
\label{2908}
\int &
\bigl( \Delta \psi -  \log \det  \bigl( {\rm Id} + D^2 \psi \bigr) \bigr) g d \gamma   \le \sqrt{\int \frac{|\nabla g|^2}{g} d \gamma} \cdot
W_2(g \cdot \gamma, \rho \cdot \gamma)
\nonumber \\&\le \frac{1+ \sqrt{1-\delta}}{\delta}W_2(g \cdot \gamma,  \gamma)  \sqrt{I(g)}
\le \frac{1+ \sqrt{1-\delta}}{\delta} I(g).
\end{align}
Let us estimate $C_P$. By the Brascamb--Lieb inequality (see \cite{BGL})
$$
\int f^2 e^{-\varphi} d \gamma - \Bigl( \int f e^{-\varphi} d \gamma  \Bigr)^2 \le
\int \langle \bigl(  {\rm Id} + D^2 \varphi \bigr)^{-1} \nabla f, \nabla f \rangle e^{-\varphi} d \gamma.
$$
Hence
\begin{align*}
\int f^2 e^{-\varphi} d \gamma - \Bigl( \int f e^{-\varphi} d \gamma  \Bigr)^2
& \le \int \bigl\| ({\rm Id} + D^2 \varphi \bigr)^{-1}\|_{op}  e^{-\varphi} d \gamma \cdot \|\nabla f\|^2_{L^{\infty}(e^{-\varphi} \cdot \gamma)}
\\&
=
 \int  \|{\rm Id} + D^2 \psi \|_{op} g d\gamma \cdot \|\nabla f\|^2_{L^{\infty}(e^{-\varphi} \cdot \gamma)}.
\end{align*}
Since $e^{-\varphi} \cdot \gamma$ is a log-concave measure, it follows from a result
 of E.~Milman on equivalence of the
isoperimetric and concentration inequalities  (\cite{Milman},
\cite[Theorem 1.5]{Milman2} or \cite[Theorem 8.7.1]{BGL}) that
$$
C_P \le c  \int  \|{\rm Id} + D^2 \psi \|_{op} g d \gamma
$$
for some universal constant $c$.
It follows from (\ref{2908}) that
$$
\int \bigl(\|D^2\psi\|_{op} - \log\det (I + \|D^2\psi\|_{op})\bigr) gd\gamma
\le \frac{1+ \sqrt{1-\delta}}{\delta} I(g).
$$
Applying the inequality $\log(1+t)\le 2^{-1} + 2^{-1}t$, we observe that
$$
\int \bigl(\|D^2\psi\|_{op} - \log\det (I + \|D^2\psi\|_{op})\bigr) g d \gamma
\ge -1/2 +1/2 \int \|D^2\psi\|_{op}\, g
d\gamma.
$$
Hence for some universal constant $C$ we have
$$
C_P \le C\Bigl(1 + \frac{1+ \sqrt{1-\delta}}{\delta} I(g) \Bigr).
$$
It remains to note that for large values of $C_P$ one has
$
\delta \sim \frac{\log C_P}{C_P}.
$
This immediately implies the announced bound.
\end{proof}

Finally, Theorem \ref{poin}, Lemma \ref{distcontrol}, and Theorem \ref{FILest}
imply our main result.

 \begin{theorem}
 \label{main}
    Assume that $g$ is a probability density such that
  $
  I(g) < \infty.
  $
   Then there exists a constant $C>0$ depending only on $I(g)$
   such that
   $$
   \mathcal{F}_{\gamma} \ge - C
   $$
   and
   $$
   W^2_2(g \cdot \gamma, \rho \cdot \gamma) \le C,
   $$
   where $\rho \cdot \gamma$ is the minimum point
   of $\mathcal{F}_{\gamma}$ such that $\rho \cdot \gamma$ has zero mean.
   \end{theorem}

Yet another result can be obtained under the
uniform bound for the Hessian of $-\log g$ by applying
the same techniques as in the proof of the Caffarelli contraction theorem.
We do not give  the full proof here (see, for instance, \cite{Kol-contr}),
but only explain the main idea.

\begin{theorem}
\label{sdb}
Let $-D^2 \log g \le c \cdot {\rm Id}$, $c > -1$.
Then $C_P \le {1+c}$ and
 $
   \mathcal{F}_{\gamma} \ge - C
   $,
   $
   W^2_2(g \cdot \gamma, \rho \cdot \gamma) \le C,
   $
   for some constant $C$ depending on $c$.
\end{theorem}
{\bf Sketch of the proof.} According to the Brascamb--Lieb inequality
$$
\int f^2 d\mu - \Bigl( \int f d \mu\Bigr)^2 \le
\int \langle \bigl(D^2 \varphi + {\rm Id}\bigr)^{-1} \nabla f, \nabla f \rangle
d \mu, \ \mu = \rho \cdot \gamma.
$$
Hence it is sufficient to show that $(D^2 \varphi + {\rm Id} )^{-1} \le (1+ c) {\rm Id} $,
or, equivalently, $D^2 \psi + {\rm Id} \le (1+ c) {\rm Id} $, where $\psi$ is the dual
potential.
This estimate can be obtained by the standard maximum principle and differentiation
of the Monge--Amp{\`e}re equation.
The maximum principle is applied in the situation
$$
1 + \psi_{ee} = \Psi_{ee},\quad
\Psi(x) = \frac{|x|^2}{2} + \psi(x) + c(n),
$$
where $e$ is a fixed unit vector.
Note that $\Psi$ satisfies the Monge--Amp{\`e}re equation
$$
\Phi(\nabla \Psi)  - \log \det D^2 \Psi = \frac{x^2}{2} - \log g  + c'(n),
$$
where $\Phi = \frac{|x|^2}{2} + \varphi$.
Let us differentiate twice the equation
\begin{multline}
\label{ma2}
\langle \nabla \Phi(\nabla \Psi), \nabla \Psi_{ee} \rangle
+ \langle D^2 \Phi(\nabla \Phi) \nabla \Psi_{e}, \nabla \Psi_{e}   \rangle  - {\rm Tr}\bigl[ (D^2 \Psi)^{-1} D^2 \Psi_{ee}  \bigr]
\\
+   {\rm Tr}\bigl[ (D^2 \Psi)^{-1} D^2 \Psi_{e}  \bigr]^2  = 1 - ( \log g )_{ee}.
\end{multline}
At every local maximum point $x_0$ of the function $\Psi_{ee}$ one has
$\nabla \Psi_{ee}=0,  D^2 \Psi_{ee} \le 0$. Note, moreover, that
$$
\langle D^2 \Phi(\nabla \Phi) \nabla \Psi_{e}, \nabla \Psi_{e}   \rangle =
\Psi_{ee}.
$$
From (\ref{ma2}) we obtain
$$
1 + \psi_{ee}=
\Psi_{ee} \le 1+ c,
$$
which completes the proof.

\subsection{High power and exponential integrability}

In this subsection we establish a priori bounds for the entropy- and
information-type integrals
$$
\int |\log \rho|^p \rho d \gamma
$$
and
$$
\int \Bigl| \frac{\nabla \rho}{\rho}\Bigr|^{p}\rho  d \gamma .
$$
 Here again  $\rho \cdot \gamma = e^{-\varphi} d \gamma$ is the (unique)
 minimum point of $\mathcal{F}_{\gamma}$
with zero mean.
Several results of this type have been obtained in
\cite{BoKo} for the standard Monge--Kantorovich problem.
The proofs of the theorems of this subsection follow the ideas from \cite{BoKo},
but here they are simpler because we benefit from the specific properties of our problem.

\begin{theorem}
Assume that $I(g) < \infty$.
Then
$$
\int \varphi^2 e^{-\varphi} d \gamma  = \int (\log \rho)^2 \rho dx < \infty.
$$
Assume, in addition, that $g$ satisfies the Poincar{\'e} inequality with  a constant $C$.
Then for every $p>0$ there exists a
number $c$ depending on $p, I(g), C$ such that
$$\int |\nabla \varphi|^p e^{-\varphi} d \gamma =
\int \Bigl| \frac{\nabla \rho}{\rho}\Bigr|^{p}\rho  d \gamma \le c $$
and
$$
\int |\varphi|^p e^{-\varphi} d \gamma  = \int |\log \rho|^p \rho d \gamma \le c.
$$
\end{theorem}
\begin{proof}
The assumptions of the theorem imply that $\mu$ satisfies the Poincar{\'e}
inequality (see Theorem \ref{poin}).
Next we note that the $\gamma$-integrability of
$|\nabla \varphi|^p e^{-\varphi}$
implies the $\gamma$-integrability of
$| \varphi|^p e^{-\varphi} d \gamma$.
This follows from the Poincar{\'e} inequality
(see an explanation in \cite[formula (1.3)]{BoKo})
and the $\gamma$-integrability of
$\varphi e^{-\varphi}$, i.e., the existence of ${\rm Ent} (e^{-\varphi})$.

Then our first claim follows immediately
from the finiteness of
$$
\int |\nabla \varphi|^{2} e^{-\varphi} d \gamma  = W_2^2(\rho \cdot \gamma, g \cdot \gamma).
$$

We now proceed by induction and assume that the theorem is proved for $p=2m$.
Let us show how to prove the claim for $2m +1$.
By the Kantorovich duality identity (see \cite{BoKo2012})
$$
\frac{|\nabla \varphi|^2}{2} + \varphi + \psi(x + \nabla \varphi)=0.
$$
Hence
$$
\int |\nabla \varphi|^{2m+1} e^{-\varphi} d \gamma
 = -\int \varphi |\nabla \varphi|^{2m-1} e^{-\varphi} d \gamma
-\int \psi(x + \nabla \varphi) |\nabla \varphi|^{2m-1} e^{-\varphi} d \gamma.
$$
We estimate the right-hand side by
$$
c(m) \Bigl(
\int |\nabla \varphi|^{2m}e^{-\varphi}d \gamma  +  \int |\varphi|^{2m} e^{-\varphi}d \gamma +
\int |\psi(x + \nabla \varphi)|^{2m} e^{-\varphi} d \gamma \Bigr).
$$
The integrals
$$
 \int |\varphi|^{2m} e^{-\varphi}d \gamma,  \ \int |\nabla \varphi|^{2m} e^{-\varphi}d \gamma
$$
are  bounded by a constant depending on $C,m, I(g)$ by the inductive assumption.
Since $g \cdot \gamma$ satisfies the Poincar{\'e} inequality,
it remains to show the integrability of $|\nabla \psi|^{2m} g$.
But the integral of this function against $\gamma$ equals the integral
of $|\nabla \varphi|^{2m} e^{-\varphi}$.
So the claim is proved for $p=2m+1$.
Repeating the arguments we  prove the assertion for $p=2m+2$.
\end{proof}

We close
this section with a result on the
exponential integrability of $|\nabla \varphi|^2$. We apply the infimum-convolution inequality,
which is known to be  another form of the transportation inequality (see \cite{BGL}):
\begin{equation}
\label{infconv}
\int e^{-f} d \gamma \le e^{\int f^* d \gamma},
\end{equation}
where
$$
f^{*}(y) = - \inf_{x} \bigl(f(x) + \frac{1}{2} |x-y|^2\bigr).
$$

The duality identity and
(\ref{infconv}) immediately imply that for every $0 \le \delta \le 1$ one has
\begin{align*}
\int e^{\frac{\delta}{2} |\nabla \varphi|^2} e^{-(1-\delta) \varphi}d \gamma
&= \int e^{{ - \delta \psi(x + \nabla \varphi)}} e^{-\varphi} d \gamma
= \int e^{-\delta \psi} g d \gamma
\\& \le \Bigl( \int e^{-\psi} d \gamma \Bigr)^{\delta}
\Bigl( \int g^{\frac{1}{1-\delta}} d \gamma\Bigr)^{1-\delta}
\le e^{\delta \int \varphi d \gamma}
\Bigl( \int g^{\frac{1}{1-\delta}} d \gamma\Bigr)^{1-\delta}.
\end{align*}
In particular, we obtain
 the following result
(note that unlike all other results in this paper we do not assume  that  $I(g) < \infty$).

\begin{theorem}
\label{exp}
Assume that $g \le C$ and $\varphi\in L^1(\gamma)$. Then
$$\int \exp\Bigl(\frac{1}{2} |\nabla \varphi|^2\Bigr)  d \gamma
\le C \exp\Bigl(\int \varphi d \gamma\Bigr).
$$
\end{theorem}

\begin{remark}
\rm
The assumption of integrability of $\varphi$
may seem quite innocent, but it is not. For instance, if $\rho$
vanishes outside a compact set, then the integral of
$\varphi$ is infinite.
On the other hand, if $\varphi$ is defined $\gamma$-a.e., then by the Cheeger inequality
$$
\int |\varphi - med| d \gamma \le 2 \int |\nabla \varphi| d \gamma,
$$
where $med$ is the median of $\varphi$. Then  it follows immediately from Theorem \ref{exp}
that the integral of $\exp\Bigl(\frac{1}{2} |\nabla \varphi|^2\Bigr)$
is bounded by a constant depending on $C, med$.
\end{remark}

\end{document}